\documentclass[11pt]{amsart}
\usepackage{pdflscape}
\usepackage{amsmath, amscd, amssymb, euler}
\usepackage[frame,cmtip,arrow,matrix,line,graph,curve]{xy}
\usepackage{graphpap, color,bbm, verbatim}
\usepackage[mathscr]{eucal}
\usepackage[normalem]{ulem}
\usepackage{pgf,tikz}
\usetikzlibrary{arrows}

\numberwithin{equation}{section}

\newcommand{\comments}[1]{}

\newcommand{\CC}{\mathbb{C}}

\newcommand{\PP}{\mathbb{P}}
\newcommand{\QQ}{\mathbb{Q}}

\newcommand{\bB}{\mathbf{B}}

\newcommand{\bM}{\mathbf{M}}

\newcommand{\cal}{\mathcal}

\def\cA{{\cal A}}

\def\cE{{\cal E}}
\def\cF{{\cal F}}

\def\cO{{\cal O}}

\def\cI{{\cal I}}

\def\lra{\longrightarrow}

\newcommand{\ses}[3]{0\rightarrow{#1}\rightarrow{#2}\rightarrow{#3}\rightarrow 0}

\def\and{\quad{\rm and}\quad}

\def\and{\quad\text{and}\quad}

\def\PP{\mathbb{P}}
\def\CC{\mathbb{C}}

\def\lra{\longrightarrow}

\def\cO{\mathcal{O}}

\def\lr{\rightarrow}

\input xy
\xyoption{all}

\DeclareMathOperator{\Ext}{Ext} 
\DeclareMathOperator{\Hom}{Hom}

\newtheorem{prop}{Proposition}[section]
\newtheorem{theo}[prop]{Theorem}
\newtheorem{lemm}[prop]{Lemma}
\newtheorem{coro}[prop]{Corollary}

\theoremstyle{definition}

\newtheorem{defi}[prop]{Definition}
\newtheorem{rema}[prop]{Remark}

\title[wall-crossings of stable pairs]{Moduli Spaces of $\alpha$-stable Pairs and Wall-Crossing on $\mathbb{P}^2$}

\author{Jinwon Choi}
\address{Department of Mathematics, Sookmyung Women's University, Seoul 140-742, Korea}
\email{jwchoi@sookmyung.ac.kr}

\author{Kiryong Chung}
\address{School of Mathematics, Korea Institute for Advanced Study, Seoul 130-722, Korea}
\email{krjung@kias.re.kr}


\begin{document}

\begin{abstract}
We study the wall-crossing of the moduli spaces $\mathbf{M}^\alpha (d,1)$ of $\alpha$-stable pairs with linear Hilbert polynomial $dm+1$ on the projective plane $\mathbb{P}^2$ as we alter the parameter $\alpha$. When $d$ is $4$ or $5$, at each wall, the moduli spaces are related by a smooth blow-up morphism followed by a smooth blow-down morphism, where one can describe the blow-up centers geometrically. As a byproduct, we obtain the Poincar\'e polynomials of the moduli spaces $\mathbf{M}(d,1)$ of stable sheaves. We also discuss the wall-crossing when the number of stable components in Jordan-H\"{o}lder filtrations is three.
\end{abstract}
\maketitle
\section{Introduction}

\subsection{Motivation and Results}
In moduli theory, for a given quasi-projective moduli space $\bM_0$, various compactifications stem from the different view points for the moduli points of $\bM_0$.
After we obtain various compactified moduli spaces of $\bM_0$, it is quite natural to ask the geometric relationship among them. Sometimes, this question is answered by birational morphisms between them, which enables us to obtain some geometric information (for example, the cohomology groups) of one space from that of the other \cite{Thad2,CKH}.

In this paper, we study the moduli space of semistable sheaves of dimension one on smooth projective surfaces \cite{simp}, which recently gains interests in both mathematics and physics.
This is an example of compactifications of the relative Jacobian variety, where we regard its general point as a sheaf
$$
F:=\cO_C(\sum_{i=1}^n p_i)
$$
on a smooth curve $C$ with pole along points $p_i$ of general position. In general, the moduli space of semistable sheaves is hard to study due to the lack of geometry of its boundary points. However, if $n$ is equal to the genus of $C$, the sheaf $F$ has a unique section up to scalar. So, we may alternatively consider the general point as a sheaf with a section, which in turn leads to another compactification, so called the moduli space of $\alpha$-semistable pairs (more generally, the coherent systems \cite{lepot3}). When $\alpha$ is large, it can be shown that the moduli spaces of $\alpha$-stable pairs are nothing but the relative Hilbert schemes of points on curves. The main advantage of this viewpoint is that in many cases the relative Hilbert scheme is more controllable than the moduli space of stable sheaves. In this paper, we are interested in comparing various compactifications and getting geometric information of the moduli space of sheaves from the relative Hilbert scheme.

We begin by reviewing the theory of $\alpha$-stable pairs. Let $X$ be a smooth projective variety with fixed ample line bundle $\cO_X(1)$. By definition, a pair $(s, F)$ consists of a sheaf $F$ on $X$
and one-dimensional subspace $s \subset H^0(F)$. Let us fix $\alpha \in \QQ[m]$ with a positive leading coefficient. A pair $(s,F)$ is called \emph{$\alpha$-semistable} if $F$ is pure and for any proper nonzero subsheaves $F'\subset  F$, the inequality
$$
\frac{\chi(F'(m))+\delta\cdot\alpha}{r(F')} \leq \frac{\chi(F(m))+\alpha}{r(F)}
$$
holds for $m\gg 0$. Here $r(F)$ is the leading coefficient of the Hilbert polynomial $\chi(F(m))$ and $\delta=1$ if the section $s$ factors through $F'$ and $\delta=0$ otherwise.
When the strict inequality holds, $(s,F)$ is called $\alpha$-stable.

The moduli space of $\alpha$-semistable pairs (more generally, \emph{coherent systems}) on a smooth projective variety was extensively studied by Le Potier \cite{lepot1,lepot2}. By general results of the geometric invariant theory, Le Potier proved that there exist projective schemes $\bM_X^{\alpha}(P(m))$ which parameterize the
\emph{$S$-equivalence classes} of $\alpha$-semistable pairs with fixed Hilbert polynomial $P(m)$ on $X$. Here we say that two $\alpha$-semistable pairs are $S$-equivalent if two pairs have equivalent Jordan-H\"{o}lder filtrations. M. He \cite{mhe} studied the geometry of moduli space of $\alpha$-stable pairs on the projective plane $X=\PP^2$ in order to compute the Donaldson numbers.


From now on,
we will denote $$\bM^{\alpha}(d,\chi):=\bM_{\PP^2}^{\alpha}(dm+\chi)$$ for $X=\PP^2$ with linear Hilbert polynomial $P(m)=dm+\chi$.

When $\alpha$ is sufficiently large (for example, $\mbox{deg}(\alpha)\geq \dim X$), $\alpha$-stable pairs are precisely stable pairs in the sense of Pandharipande-Thomas \cite{pt}. 
Moreover, when $X$ is $\PP^2$, we have the following.

\begin{prop}[{\cite[\S 4.4]{mhe}, \cite[Proposition B.8]{PT}}]\label{prop0}\hfill\\
If $\alpha$ is sufficiently large, then $\bM^{\alpha:=\infty}(d,\chi)$ is
isomorphic to the relative Hilbert scheme of points on the universal degree $d$ curve. Moreover, it is an \emph{irreducible normal} variety.
\end{prop}
Here, the number of points on the relative Hilbert scheme is given by $$n:=\chi-\frac{d(3-d)}{2}.$$ We will denote by $\bB(d,n)$ the relative Hilbert scheme of $n$ points on the universal degree $d$ curve.

At the other extreme when $\alpha$ is sufficiently small, the moduli space has a natural forgetful morphism into the moduli space of semistable sheaves, so called the \emph{Simpson space} \cite{simp}. We denote by $\bM(d,\chi)$ the moduli space of semistable sheaves on $\PP^2$ with Hilbert polynomial $dm+\chi$.
Sometimes, we identify the space $\bM(d,\chi)$ with the moduli space of pairs with a zero section.
\begin{prop}\label{prop4}
If $\alpha$ is sufficiently small (denoted by $\alpha=0^+$), there is a natural morphism
$$
\xi: \bM^{0^+}(d,\chi) \lra \bM(d, \chi)
$$
which associates to the $0^+$-stable pair $(s,F)$ the sheaf $F$.
\end{prop}
When $\chi=1$, a general stable sheaf has a unique section up to a scalar multiplication. So, the moduli spaces $\bM^{\alpha}(d,1)$ and $\bM(d,1)$ are birational.
Now we state the main problem of this paper.

\bigskip
\noindent\textbf{Problem:} Compare the moduli spaces $\bM^{\alpha}(d,1)$ and $\bM(d,1)$ by using birational morphisms and compute the cohomology group of $\bM(d,1)$.
\bigskip

Note that, when $d\leq 3$, all moduli spaces $\bM^{\alpha}(d,1)$ and $\bM(d,1)$ are isomorphic to each other \cite{lepot1}.
In this paper, we answer this problem for $d=4$ and $5$.
\begin{theo}[Theorem \ref{mainthm1} and Theorem \ref{mainthm2}]
Assume $d=4$ or $5$.
\begin{enumerate}
\item The moduli space $\bM^{0^+}(d,1)$ is obtained from the $\infty$-stable pair space $\bM^{\infty}(d,1)$ by several wall-crossings such that each wall-crossing is a composition of a smooth blow-up morphism followed by blow-down one.
\item The forgetful map $\xi:\bM^{0^+}(d,1) \lr \bM(d,1)$ is a divisorial contraction such that the exceptional divisor can be described by stable pair spaces with various Euler characteristics.
\end{enumerate}
\end{theo}
As corollaries, we obtain the Poincar\'e polynomials of $\bM(d,1)$ by using those of the relative Hilbert schemes of points on the plane curves (\S5).

To prove part (1), we first find the flipping locus at each wall by using the
stability conditions. It turns out that the blow-up centers can be described as a configuration of points on curves. In particular, they are projective bundles over the product of the moduli spaces of $\alpha$-stable pairs of lower degrees.
After blowing up the moduli space along such loci, by performing the \emph{elementary modifications of pairs} (Definition \ref{elem}), we construct a flat family of pairs that are stable on the other side of the wall, which in turn gives a birational morphism. This morphism is shown to be a smooth blow-down morphism by analyzing the exceptional divisor and applying the Fujiki-Nakano criterion \cite{Fujiki}.

For the part (2), when $d=4$, it can easily be checked that the forgetful morphism
$\xi$ is a divisorial contraction by using the classification of
stable sheaves \cite{maican}. Furthermore, we show that $\xi$ is a smooth blow-up morphism along the Brill-Noether locus (Proposition \ref{prop5}). 

If $d=5$, by using the classification of stable sheaves we can check that the Brill-Noether locus consists of two strata, where the smaller dimensional one is the boundary of the bigger one. Moreover, the whole Brill-Noether locus is an irreducible variety which can easily be obtained from the wall-crossing of $\bM^{\alpha}(5,-1)$ (For detail, see \S4).

\subsection{Outline of the Paper} The stream of this paper is as follows. In \S 2, we review well-known facts about the deformation theory of pairs \cite{mhe} and the notion of the elementary modification of pairs. In \S 3, we provide a proof of Theorem \ref{mainthm1} and Theorem \ref{mainthm2}, that is, we compare the moduli space of $\alpha$-stable pairs by wall-crossing when $d=4$ and $5$. In \S4, we study the forgetful morphisms geometrically and the Brill-Noether loci.
As a corollary, in \S 5, we obtain the Poincar\'e polynomials of $\bM(4,1)$ and $\bM(5,1)$. In \S6, we discuss the wall-crossing for $\bM^{\alpha}(4,3)$ when the number of terms in the Jordan-H\"{o}lder filtration is more than two.


\subsection{Further Works}
For the case $d\geq 6$, the moduli spaces $\bM^{\infty}(d,1)$ do not have a bundle structure over the Hilbert scheme of points and thus we can not apply the same method to calculate the Betti numbers of $\bM(d,1)$. However, one can still compute the topological Euler characteristics of $\bM^{\infty}(d,1)$ by means of the torus localization \cite{ptvertex}. Moreover, under the assumption that the Joyce-Song-type wall-crossing formula \cite{js} holds, the first author has computed the Euler characteristics of $\bM(d,1)$ up to degree $10$ and verified that the results agree with the prediction in physics \cite{choithesis}.

One can go further. In \cite{ckk}, new definition of the \emph{refined Pandharipande-Thomas invariant} is proposed via an extension of the Bia{\l}ynicki-Birula decomposition to singular moduli spaces. It is ``refined'' partly in the sense that it is the virtual motive of $\bM^{\infty}_X(d,\chi)$, which specializes to the Pandharipande-Thomas invariant \cite{pt}. A product formula for these refined Pandharipande-Thomas invariants is conjectured in \cite{ckk}, which is consistent with B-model calculation in physics. It is expected that the lower degree correction terms produced from the product formula are in correspondence with the wall-crossing terms in our paper. Details can be found in \cite{ckk}.

\medskip
\textbf{Acknowledgement.}
We would like to thank Sheldon Katz, Young-Hoon Kiem, Wanseok Lee, and Han-Bom Moon for valuable discussion and comments. We also thank the anonymous reviewers for their valuable comments and suggestions to improve the quality of the paper. 

\section{Preliminaries}
In this section, we collect well-known properties of pairs: deformation theory and the elementary modification.
\subsection{Deformation Theory of Pairs}
Let $X$ be a smooth projective variety.
Deformation theory of pairs (more generally, coherent systems) on $X$ was studied in \cite{mhe,lepot2}. We summarize the results for convenience of readers. Let $\bM$ be the moduli space of semistable pairs on $X$.
We note that the set of all coherent systems forms an abelian category. Also, the category of coherent systems has enough injective objects, so it is possible to define the $\Ext^i(\Lambda, \Lambda')$ of coherent systems $\Lambda$, $\Lambda'$. We consider the category of pairs as its subcategory.
\begin{prop}\cite[Corollary 3.10, Theorem 3.12]{mhe}\label{depair}
The first order deformation space of $\bM$ at a stable pair $\Lambda$ on a smooth projective variety $X$ is given by 
$$
\Ext^1(\Lambda, \Lambda).
$$
Moreover, if $\Ext^2(\Lambda,\Lambda)=0$, then $\bM$ is smooth at $\Lambda$.
\end{prop}
We will use the following proposition repeatedly in \S3 and \S4.
\begin{prop}\cite[Corollary 1.6]{mhe}\label{defcoh}
Let $\Lambda=(s, F)$ and $\Lambda'=(s', F')$ be pairs on $X$. Then there is a long exact sequence
\begin{align*}
0&\lr \Hom(\Lambda,\Lambda')\lr \Hom (F,F')\lr \Hom(s,H^0(F')/s')\\
&\lr \Ext^1(\Lambda,\Lambda')\lr \Ext^1(F,F')\lr \Hom(s,H^1(F'))\\
&\lr \Ext^2(\Lambda,\Lambda')\lr \Ext^2(F,F')\lr \Hom(s,H^2(F'))\lr \cdots.
\end{align*}
\end{prop}
\begin{lemm}\label{lem0}
If $\chi<\frac{4+5d-d^2}{2}$, the moduli spaces $\bM^{\infty}(d,\chi)$ are projective bundles over Hilbert scheme of points on $\PP^2$. Specially, they are smooth.
\end{lemm}
\begin{proof}
Let $n=\chi-\frac{d(3-d)}{2}$. Then by Proposition \ref{prop0}, $\bM^{\infty}(d,\chi)$ is isomorphic to $\bB(d,n)$. A closed point $(C,Z)$ in $\bB(d,n)$ can be considered as a choice of a section of the ideal sheaf $I_Z(d)$ \cite[\S 4.4]{mhe}.

We have the canonical projection $q\colon \bB(d,n)\to Hilb^{n}(\PP^2)$. Let $\cI$ be a universal ideal sheaf on $Hilb^n(\PP^2)\times \PP^2$ and $p$ be the projection to the first factor. Then, $\bB(d,n)$ is the projective bundle $\PP({p}_* \cI(d)^*)$, provided that ${p}_* \cI(d)$ is locally free. If $n\leq d+1$, we have $H^1(I_Z(d))=0$ for any length $n$ subscheme $Z$ of $\PP^2$ because the line bundle $\cO_{\PP^2}(d)$ is $d$-very ample. Hence, by semicontinuity theorem, ${p}_* \cI(d)$ is locally free as required.
\end{proof}

\begin{rema}\label{rem1}
\begin{enumerate}
\item The bound in Lemma \ref{lem0} is sharp. For example, it can be easily seen that $\bB(6,8)\simeq \bM^\infty(6,-1)$ is not smooth. In fact, for a stable pair $\Lambda$ with Hilbert polynomial $6m-1$, the obstruction space $\Ext^2(\Lambda,\Lambda)$ may not vanish.
\item Under the assumption in Lemma \ref{lem0}, let $\Lambda\in \bM^{\infty}(d,\chi)$ be a stable pair. One can easily check that
\begin{equation}\label{eq0-1}
\Ext^2(\Lambda,\Lambda)=0
\end{equation}
by Proposition \ref{defcoh} and the constancy of the Euler form.
\end{enumerate}
\end{rema}


\subsection{Elementary Modification of Pairs}
We introduce the notion of the modification of pairs on a smooth projective variety $X$. This is a main tool to relate various moduli spaces of semistable pairs by birational morphisms.
\begin{defi}\label{elem}
Let $\cF$ be a flat family of pairs on $X\times S$ parameterized by a smooth variety $S$. Let $\Delta$ be a smooth divisor of $S$ such that the restricted pair $\cF|_{X\times \Delta}$ over $\Delta$ has a flat family $\cA$ of destabilizing quotient pairs. Then we call
$$
elm_{\Delta} (\cF, \cA):=\mbox{ker}(\cF\rightarrow \cF|_{X\times \Delta}\twoheadrightarrow \cA)
$$
the \emph{elementary modification} of pair $\cF$ along $\Delta$.
\end{defi}
In general, elementary modification of pairs interchanges the subpair with the quotient pair. For example, see the proof of \cite[Lemma 4.24]{mhe}.
\section{Wall-Crossing among Moduli Spaces of $\alpha$-stable Pairs}
In this section, we will compare $\infty$-stable pair space $\bM^{\infty}(d,1)$ with the $0^+$-stable pair space $\bM^{0^+}(d,1)$ for $d=4$ and $5$ by using wall-crossing.
By the variation of geometric invariant theoretic quotients \cite{Thad,DH} and the construction of semistable pair space \cite{lepot3}, it seems to be clear there are flipping spaces
among $\bM^{\infty}(d,1)$ and $\bM^{0^+}(d,1)$ in a broad sense. In the following two subsections, we will show that the flipping spaces are related by smooth blow-up and followed by smooth blow-down morphisms (Theorem \ref{mainthm1} and Theorem \ref{mainthm2}). By using the same technique, we also relate the $\alpha$-stable pair spaces with Hilbert polynomial $5m-1$ by smooth blow-up/down morphisms. In Theorem \ref{prop3}, we will present the results without proof. This results will be essential to compute the Betti numbers of the moduli space $\bM(5,1)$ (Corollary \ref{coro2}).
\subsection{Wall-Crossing for $d=4$}
The aim of this subsection is to provide a proof of Theorem \ref{mainthm1} below. We prove that the moduli spaces of $\alpha$-stable pairs with Hilbert polynomial $4m+1$ are related by single blow-up/down morphisms. The main tool of the proof is the elementary modification of pairs and the Fujiki-Nakano criterion \cite{Fujiki}.
We start with a geometric description of walls.
\begin{lemm}\label{lem2}
We have a unique wall at $\alpha=3$ where the strictly semistable points are of type
\begin{equation}\label{eq1-8}
(1,(3,0))\oplus (0,(1,1)),
\end{equation}
where $(1, (d,\chi))$ (resp. $(0, (d,\chi))$) denote the pair $(s,F)$ with a nonzero (resp. zero) section $s$ and Hilbert polynomial $\chi(F(m))=dm+\chi$.
\end{lemm}
\begin{proof}
By \cite[Theorem 4.2]{mhe}, the wall occurs at the values of $\alpha$ for which there exist strictly $\alpha$-semistable pairs. This lemma is a consequence of an elementary calculation.
\end{proof}
Let $\Omega$ be the flipping locus in $\bM^{\infty}(4,1)$, that is, the inverse image of the locus of strictly semistable pairs along the natural map $ \bM^{\infty}(4,1) \lr \bM^{3}(4,1)$. The locus $\Omega$ can be described as a projective bundle as follows. Let $\bM(1,1)$ (resp. $\bM^{0^+}(3,0)$) be the moduli space of stable pairs having  a unique (resp. non) zero section with Hilbert polynomial $m+1$ (resp. $3m$). Then they have a universal family of pairs $\cF'$ (resp. $\cF''$) on $\bM(1,1)\times \PP^2$ (resp. $\bM^{0^+}(3,0)\times \PP^2$) \cite[Theorem 4.3]{mhe}.
In fact, $\bM(1,1)\cong Gr(2,3)$ and $\bM^{0^+}(3,0)\cong \PP^9$. The latter is because if $F\in \bM(3,0)$ has a nonzero section, this nonzero section defines a nonzero morphism $\cO_C\to F$ for some cubic curve $C$. Then by stability one can check this morphism is an isomorphism. See the proof of \cite[Theorem 4.4]{lepot1}.

Then one can easily construct the universal families $\cF'$ and $\cF''$.
Let
\begin{align*}
&q_1:\bM(1,1)\times\bM^{0^+}(3,0)\times \PP^2\lr \bM(1,1)\times \PP^2,\\
&q_2:\bM(1,1)\times \bM^{0^+}(3,0)\times \PP^2 \lr \bM^{0^+}(3,0)\times \PP^2, \text{and}\\
&p:\bM(1,1)\times\bM^{0^+}(3,0)\times \PP^2\lr \bM(1,1)\times\bM^{0^+}(3,0)
\end{align*}
be the projection maps. Then one can easily check that the relative Ext sheaf $\mathcal{E}xt_p^1(q_2^*\cF'',q_1^*\cF')$ on $\bM(1,1)\times \bM^{0^+}(3,0)$ is a locally free sheaf of rank $4$.
Let
$$ P:=\PP(\mathcal{E}xt_p^1(q_2^*\cF'',q_1^*\cF'))$$
be the projective bundle on $\bM(1,1)\times \bM^{0^+}(3,0)$. Then there exists a universal sheaf $\cE$ on $P\times \PP^2$ parameterizing the non-split extension sheaves in $\Ext_{\PP^2}^1(\cF''|_s, \cF'|_t)$ for $s\in \bM^{0^+}(3,0)$ and $t\in \bM(1,1)$ (\cite{lang, tomm}). Thus the map
$$
P \hookrightarrow \bM^{\infty}(4,1)
$$
given by the universal sheaf $\cE$ over $P$ is a closed embedding, whose image is precisely $\Omega$.
\begin{rema}\label{rem2}
Under the identification of the moduli space $\bM^{\infty}(4,1)$ with the relative Hilbert scheme $\bB(4,3)$ of three points on quartic curves (Proposition \ref{prop0}),
each closed point of the fiber $\PP^3$ corresponds to the length three subscheme $Z$ of a pair $(Z, L\cdot C)$ such that $Z$ lies on a line $L$ and $C$ is a cubic curve. Let $(1,F)\in P$ be a non split extension class $\ses{(0,\cO_L)}{(1,F)}{(1,\cO_C)}$ where $(0,\cO_L) \in  \bM(1,1)$ (resp. $(1,\cO_C)\in \bM^{0^+}(3,0)$) for fixed $L$ and $C$. Since $(1,F)$ is $\infty$-stable pair, there is a short exact sequence $\ses{(1,\cO_{L\cdot C})}{(1,F)}{(0,Q)}$ for some torsion sheaf $Q$ of the length three on the quartic curve $L\cdot C$. Combining these two short exact sequences and some diagram chasing,
it can be seen that the torsion sheaf $Q$ is supported on $L$ and thus $Q\cong \cO_Z$ for some $Z$ of length three subscheme of $L$.
\end{rema}
Now we will state one of the main theorems.
\begin{theo}\label{mainthm1}
Let $\bM^{\alpha}(4,1)$ be the moduli space of
$\alpha$-semistable pairs on $\PP^2$ with Hilbert
polynomial $4m+1$. Then there is
a flip diagram at $\alpha=3$
$$
\xymatrix{& \widetilde{\bM^{\infty} (4,1)}\ar[rd]^p\ar[ld]_q& \\
\bM^{\infty} (4,1) \ar[dr]& & \bM^{0^+} (4,1) \ar[dl] \\
&\bM^{3} (4,1)&}
$$
such that $\widetilde{\bM^{\infty} (4,1)}$ is the smooth blow-up of $\bM^{\infty} (4,1)$ along $\Omega$ with the exceptional divisor $\widetilde{\Omega}$ and the morphism $p:\widetilde{\bM^{\infty} (4,1)} \lr \bM^{0^+} (4,1)$ is a smooth blow-down one contracting $\widetilde{\Omega}$ along the other direction.
\end{theo}

As we saw above, the flipping locus $\Omega$ is a $\PP^3$-bundle over $\bM(1,1)\times \bM^{0^+}(3,0)$. We first describe the normal space of $\Omega$.
\begin{lemm}\label{lem1}
Let $\Lambda_1:=(1,\cO_C)$ and $\Lambda_2:=(0,\cO_L)$. Then the normal bundle of $\Omega$ in $\bM^{\infty}(4,1)$ restricted to the fiber $\PP^3=\PP(\Ext^1( \Lambda_1, \Lambda_2))$ over a point $[\Lambda_1\oplus \Lambda_2] \in \bM^{3}(4,1)$ is given by
$$
N_{\Omega/\bM^{\infty}(4,1)}|_{\PP^3}\simeq \Ext^1( \Lambda_2, \Lambda_1)\otimes \cO_{\PP^3}(-1).
$$
\end{lemm}
\begin{proof}
A pair $(1,F)\in \Omega$ fits into a non split extension
\begin{equation}\label{eq1-1}
\ses{\Lambda_2}{(1,F)}{\Lambda_1}.
\end{equation}

By Proposition \ref{depair}, we know that the first order deformation space of the pair $(1,F)$ in \eqref{eq1-1} is $\Ext^1((1,F),(1,F))$, which fits into the exact diagram:
\begin{equation}\label{eq1-4}
\xymatrix{ &\Ext^1(\Lambda_1,\Lambda_2)\ar[d]&&\Ext^1(\Lambda_1,\Lambda_1)\ar[d]\\
0\ar[r]&\Ext^1((1,F),\Lambda_2)\ar[r]\ar[d]&\Ext^1((1,F),(1,F))\ar[r]^{\phi_1}&\Ext^1((1,F),\Lambda_1)\ar[d]^{\phi_2}\\
&\Ext^1(\Lambda_2,\Lambda_2)&&\Ext^1(\Lambda_2,\Lambda_1).}
\end{equation}
In \eqref{eq1-4}, the $0$ term comes from
\begin{equation}\label{eq1-5}
\Ext^0((1,F),\Lambda_2)=0 \mbox{ and }\Ext^0((1,F),(1,F))=\Ext^0((1,F),\Lambda_1)=\CC.
\end{equation}
Since $(1,F)$, $\Lambda_1$, and $\Lambda_2$ are stable, the first two are obvious. To prove the last one, consider the long exact sequence
$$
0\rightarrow\Ext^0(\Lambda_1,\Lambda_1)=\CC\rightarrow\Ext^0((1,F),\Lambda_1)\rightarrow\Ext^0(\Lambda_2,\Lambda_1)\rightarrow\cdots,$$
which is given by taking $\Hom(-, \Lambda_1)$ to \eqref{eq1-1}.
Here the term $\Ext^0(\Lambda_2,\Lambda_1)=0$ from the slop condition applied to the stable pairs $\Lambda_i$.

Recall that $\Omega$ is a $\PP^3$-bundle over $\bM(1,1)\times \bM^{0^+}(3,0)$. So the tangent space of $\Omega$ at $(1,F)$ is isomorphic to the direct sum of the three extensions:
\begin{equation}\label{eq1-3}
\Ext^1(\Lambda_1,\Lambda_2)/\CC \simeq\CC^3, \Ext^1(\Lambda_2,\Lambda_2)\simeq\CC^2, \mbox{ and } \Ext^1(\Lambda_1,\Lambda_1)\simeq\CC^9,
\end{equation}
where each extension is the first order deformation space of $\PP^3$, $\bM(1,1)$, and $\bM^{0^+}(3,0)$ respectively.
Thus the kernel of the composite map $\phi = \phi_2 \circ \phi_1$ contains the tangent space $T_{(1,F)}\Omega$.
To prove the lemma, it is enough to check that the obstruction spaces vanish.
Once this holds, since this consideration is canonical, this holds for every point in $\PP^3$. Thus if we relativize the above diagram \eqref{eq1-4} over the projective bundle $\Omega$ one can easily see that the restricted normal bundle of $\Omega$ on each fiber $\PP^3$ is isomorphic to $\Ext^1( \Lambda_2\otimes\cO_{\PP^3}(1), \Lambda_1)\cong\Ext^1( \Lambda_2, \Lambda_1)\otimes\cO_{\PP^3}(-1)$. (\cite{tomm})

One can check that
\begin{equation}\label{eq1-2}
\Ext^2(\Lambda_i,\Lambda_j)=0
\end{equation}
for all $i ,j =1,2$. If $i=j$, this directly comes from Remark \ref{rem1} and Serre duality. If $i \ne j$, by using Proposition \ref{defcoh}, we know that it is enough to check
$$
\Ext^2(\cO_L, \cO_C)=\Ext^2(\cO_C, \cO_L)=0.
$$
But this clearly holds by Serre duality again.
\end{proof}

We remark that, by Lemma \ref{lem1}, the exceptional divisor $\widetilde{\Omega}$ of the blow-up morphism
$$q:\widetilde{\bM^{\infty} (4,1)}\lr \bM^{\infty}(4,1)$$
along $\Omega$ is a $\PP^3\times \PP^2$-bundle over $\bM(1,1)\times \bM^{0^+}(3,0)$ and
the normal bundle $\widetilde{\Omega}$ restricted to the fiber $\PP^3\times \PP^2$ is
\begin{equation}\label{eq1-7}
N_{\widetilde{\Omega}/\widetilde{\bM^{\infty} (4,1)}}|_{\PP^3\times \PP^2 }\simeq \cO(-1,-1).\end{equation}

\begin{proof}[Proof of Theorem \ref{mainthm1}]
Let the universal families $\cF'$ and $\cF''$ be as before. Let $\cF$ be a universal pair on $\bM^{\infty}(4,1) \times \PP^2$ \cite[Theorem 4.3]{mhe}. Then the restricted family $\cF|_{\Omega\times \PP^2}$ fits into the short exact sequence
$$
\ses{\cF'}{\cF|_{\Omega\times \PP^2}}{\cF''}.
$$
Let $\cF''_1:=(q|_{\widetilde{\Omega}}\times 1_{\PP^2})^* \cF''$. For each $z \in \widetilde{\Omega}$ such that $q(z)=[(1,F)]$ in \eqref{eq1-1}, we have $\cF''_1|_{\{z\}\times \PP^2}=(1,\cO_C)$ for a cubic curve $C$. Hence the pair $\cF''_1$ is a family of destabilizing quotients of the pull-back of the universal family $\cF$.

Let
$$
\widetilde{\cF}:= elm_{\widetilde{\Omega}}((q\times 1_{\PP^2})^*\cF, \cF''_1)
$$
be the elementary modification of the pull-back of $\cF$ along $\widetilde{\Omega}$.

We claim that $\widetilde{\cF}$ induces a birational morphism to $\bM^{0^+}(4,1)$.
The effect of elementary modification of pairs is the interchange of the sub/quotient of pairs
 \cite[Lemma 4.24]{mhe}. In our case,
this can be proved by analyzing the deformation space of pairs as follows (cf. \cite{CKH}). Choosing a vector $v$ in the tangent space
$$
T_{q(z)}\bM^{\infty}(4,1)=\Hom_{\CC}(\mbox{Spec}\CC[\epsilon]/(\epsilon^2), \bM^{\infty}(4,1))
$$
is the same as having a family $\cF|_{\widetilde{\PP^2}}$ restricted on $\widetilde{\PP^2}:=\mbox{Spec}\CC[\epsilon]/(\epsilon^2) \times \PP^2$
such that the central fiber is $\cF|_{\{0\}\times \PP^2}=(1,F)$ in \eqref{eq1-1}.
If $v\notin T_{q(z)}\Omega$, then the modified pair $\widetilde{\cF}|_{\widetilde{\PP^2}}$ is given by the pulling-back
$$
\xymatrix{ 0\ar[r]&\epsilon \cdot \cF|_{\widetilde{\PP^2}}\ar@{=}[d]\ar[r]&\widetilde{\cF}|_{\widetilde{\PP^2}}\ar@{-->}[d]\ar[r]&(0,\cO_L)\ar[d]\ar[r]&0\\
0\ar[r]&\epsilon \cdot \cF|_{\widetilde{\PP^2}} \ar[r]&\cF|_{\widetilde{\PP^2}}\ar[r]&(1,F)\ar[r]&0,
}
$$
where the right vertical arrow comes from \eqref{eq1-1}. Moreover, the central fiber
$$
\widetilde{\cF}|_{\widetilde{\PP^2}}/\epsilon \cdot \widetilde{\cF}|_{\widetilde{\PP^2}}
$$
is given by the push-out diagram.
$$
\xymatrix{ 0\ar[r]&\epsilon \cdot (1,\cO_C)\ar[r]&\widetilde{\cF}|_{\widetilde{\PP^2}}/\epsilon \cdot \widetilde{\cF}|_{\widetilde{\PP^2}}\ar[r]&(0,\cO_L)\ar[r]&0\\
0\ar[r]&\epsilon \cdot \cF|_{\widetilde{\PP^2}} \ar[r]\ar[u]&\widetilde{\cF}|_{\widetilde{\PP^2}}\ar[r]\ar@{-->}[u]&(0,\cO_L)\ar@{=}[u]\ar[r]&0,
}
$$
where the left vertical arrow comes from \eqref{eq1-1}.
These operations are explained as the following $\CC$-linear map
$$
KS:T_{q(z)}\bM^{\infty}(4,1)\simeq\Ext^1((1,F),(1,F)) \lr \Ext^1((0,\cO_L),(1,F))\lr \Ext^1((0,\cO_L),(1,\cO_C)),
$$
which associates $v \in T_{q(z)}\bM^{\infty}(4,1)$ to $\widetilde{\cF}|_{\{z\}\times \PP^2}$ for $v\neq 0$ of $z\in \widetilde{\Omega}$.
Note that the first isomorphism is the Kodaira-Spencer map and the others are from \eqref{eq1-1}.

On the other hand, by the proof of Lemma \ref{lem1}, the kernel of the map $KS$ is isomorphic to the tangent space $T_{q(z)}\Omega$ at $q(z)$.
Thus the modified sheaf along the normal direction of $\Omega$ is exactly a non split extension class in $\Ext^1((0,\cO_L),(1,\cO_C))$, which turns out to be a $0^+$-stable stable pair by direct calculation.
Hence there is a birational morphism
\begin{equation}\label{eq1-6}
p:\widetilde{\bM^{\infty} (4,1)}\lr \bM^{0^+}(4,1)
\end{equation}
associated to $\widetilde{\cF}$ by the universal property of the moduli space $\bM^{0^+}(4,1)$.

Now, we show that the morphism $p$ in \eqref{eq1-6} is a smooth blow-down contracting the $\PP^3$-direction of the $\widetilde{\Omega}$.
Clearly, the image of $\widetilde{\Omega}$ along the map $p$ is exactly the flipping locus in $\bM^{0^+}(4,1)$ and $p$ contracts the fibers $\PP^3$. So, to apply Fujiki-Nakano criterion \cite{Fujiki}, it is enough to check that
\begin{enumerate}
\item the restricted normal bundle of $\widetilde{\Omega}$ to a fiber $\PP^3$ is $\cO(-1)$ and
\item the space $\bM^{0^+}(4,1)$ is smooth.
\end{enumerate}
Part (1) directly comes from \eqref{eq1-7}. For part (2), let $(1,G)$ be a $0^+$-stable pair in the flipping locus. Then the pair $(1,G)$ fits into an exact sequence
$$
\ses{\Lambda_1=(1,\cO_C)}{(1,G)}{\Lambda_2=(0,\cO_L)}.
$$
By \eqref{eq1-2}, the obstruction $\Ext^2((1,G),(1,G))=0$ as required.
\end{proof}
\subsection{Wall-Crossing for $d=5$}
The walls and the possible type of strictly semistable
pairs are given as the following table.
\begin{center}
\begin{tabular}{|l|p{6cm}|}
\hline
\multicolumn{2}{|l|}{$(d,\chi)=(5,1)$} \\
\hline
$\alpha$
&
$\Lambda_1:=(1,P(F_1))\oplus \Lambda_2:=(0,P(F_2))$\\
\hline
14&
$(1,(4,-2))\oplus (0,(1,3))$\\
\hline
9&
$(1,(4,-1))\oplus (0,(1,2))$\\
\hline
4&
$(1,(4,0))\oplus (0,(1,1))$\\
\hline
$\frac{3}{2}$&$(1,(3,0))\oplus (0,(2,1))$\\
\hline
\end{tabular}
\end{center}
Here, the Hilbert polynomial $P(F_i)=\chi(F_i(m))=dm+\chi$ is denoted by $(d,\chi)$.
\begin{rema}\label{rem3}
As in Remark \ref{rem2}, each wall-crossing locus can be described as a configuration of points on reducible quintic curves. That is, regarding $\bM^{\infty}(5,1)$ as the relative Hilbert scheme $\bB(5,6)$ of six points on quintic curves, the wall-crossing loci are (the strict transformations of) the locus of pairs of six points, five points, four points on a line with a quartic curve at the wall $\alpha=14,9,4$, respectively, and lastly six points on a conic curve with a cubic curve at $\alpha=\frac{3}{2}$. These are very similar to the wall-crossing in \cite[\S 10.5]{coskun}.
\end{rema}

\begin{theo}\label{mainthm2}
Let $\bM^{\alpha}(5,1)$ be the moduli space of
$\alpha$-semistable pairs on $\PP^2$ with Hilbert polynomial
$5m+1$. Then, we have the wall-crossing diagrams
$$
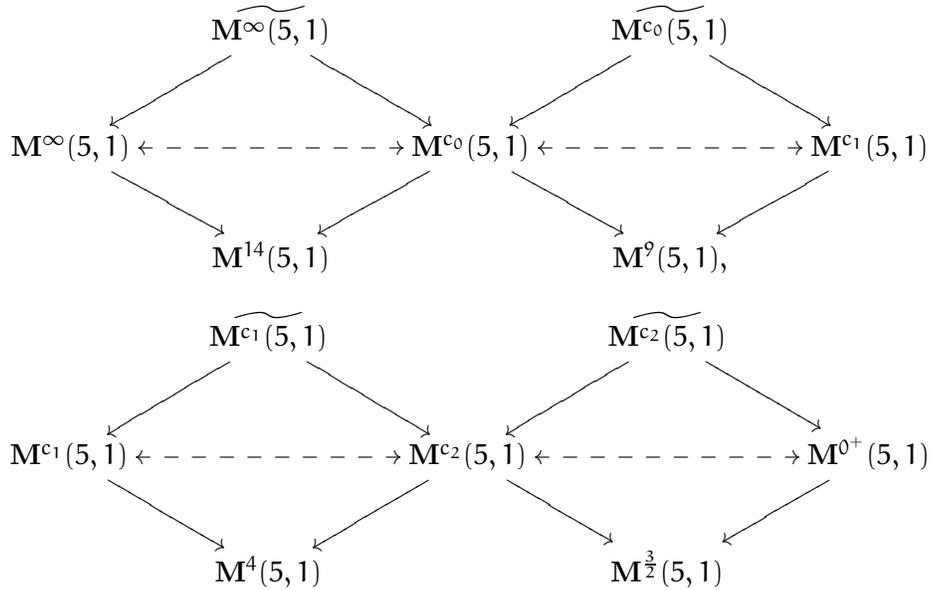
\xymatrix{& \widetilde{\bM^{\infty}(5,1)}\ar[rd]\ar[ld] &&\widetilde{\bM^{c_0}(5,1)}\ar[rd]\ar[ld]& \\
\bM^{\infty}(5,1) \ar[dr]\ar@{<-->}[rr]&&\bM^{c_0}(5,1)\ar[dr] \ar[dl]\ar@{<-->}[rr]& & \bM^{c_1}(5,1)\ar[dl]\\
&\bM^{14}(5,1)&&\bM^{9}(5,1),&}
$$
$$
\xymatrix{& \widetilde{\bM^{c_1}(5,1)}\ar[rd]\ar[ld] &&\widetilde{\bM^{c_2}(5,1)}\ar[rd]\ar[ld]& \\
\bM^{c_1}(5,1) \ar[dr]\ar@{<-->}[rr]&&\bM^{c_2}(5,1)\ar[dr] \ar[dl]\ar@{<-->}[rr]& & \bM^{0^+}(5,1)\ar[dl]\\
&\bM^{4}(5,1)&&\bM^{\frac{3}{2}}(5,1)&}
$$
where the rational numbers $\alpha$'s are $\infty>14$, $c_0 \in (9,14)$, $c_1 \in (4,9)$, $c_2 \in (\frac{3}{2},4)$ and $0^+\in (0,\frac{3}{2})$. All of upper
arrows are smooth blow-up morphisms.
\end{theo}
\begin{proof}
Let us denote the first (resp. second) stable component in table by $\Lambda_1$ (resp. $\Lambda_2$). Note that in any cases, $\Lambda_1$ and $\Lambda_2$ are $\alpha$-stable for any $\alpha$, because there is no wall for those types.

The proof is parallel to that of Theorem \ref{mainthm1}. It suffices to check the equations \eqref{eq1-5}, \eqref{eq1-3}, and \eqref{eq1-2} at each wall.

The same argument as before checks \eqref{eq1-5}.

For the vanishing of obstructions (equation \eqref{eq1-2}), recall that $\Ext^2(\Lambda_i, \Lambda_i)=0$ for all $i$ by Lemma \ref{lem0}. To check
\begin{equation}\label{eq:deg5obs}
\Ext^2(\Lambda_i, \Lambda_j)=0 \mbox{ for } i\neq j,
\end{equation}
it suffices to check
$$
\Ext^2(F_i, F_j)=0 \mbox{ for } i\neq j
$$
by Proposition \ref{defcoh}.
If $i<j$ or $\alpha<14$, then this holds obviously by Serre duality and stability. The remaining case is at $\alpha=14$.
That is, we prove $\Ext^2(\cO_L(2),\cO_C)=0$ for a quartic curve $C$ and line $L$. By Serre duality again,
$$\Ext^2(\cO_L(2),\cO_C)\simeq\Ext^0(\cO_C, \cO_L(-1))^*$$
But the latter group is zero since $$\Ext^0(\cO_C, \cO_L(-1))\subset\Ext^0(\cO, \cO_L(-1))=H^0(\cO_L(-1))=0.$$

Next we show that the first order deformation spaces have the expected dimensions (equation \eqref{eq1-3}). That is,
\begin{align}
&\Ext^1((1,F_{4m-2}),(0,\cO_L(2)))=\CC^7, &&\Ext^1((1,F_{4m-1}),(0,\cO_L(1)))=\CC^6,\label{eq2-1}\\
&\Ext^1((1,F_{4m}),(0,\cO_L))=\CC^5, && \Ext^1((1,F_{3m}),(0,\cO_Q))=\CC^7, \notag
\end{align}
where $L$ (resp. $Q$) is a line (resp. conic) and $F_{p(m)}$ is a semistable sheaf with Hilbert polynomial $p(m)$.

Let $(1,F)$ be one of $(1,F_{4m-2})$, $(1,F_{4m-1})$, or $(1,F_{4m})$. Then, $(1,F)$ fits into an exact sequence
$$\ses{(1,\cO_C)}{(1,F)}{(0,Q)},$$
for a quartic curve $C$ and a zero dimensional sheaf $Q$.
By applying $\Hom(-,\cO_L(k))$ for appropriate $k$ ($k=0, 1,\textrm{ or }2$), we get an exact sequence
\begin{align*}
0&\to \Ext^1 ((0,Q), (0,\cO_L(k)))\to \Ext^1((1,F), (0,\cO_L(k))) \to \Ext^1((1,\cO_C), (0,\cO_L(k))) \\
&\to\Ext^2 ((0,Q), (0,\cO_L(k)))\to 0,
\end{align*}
because $\Hom((1,\cO_C), (0,\cO_L(k)))$ is clearly zero and $\Ext^1((1,F), (0,\cO_L(k)))$ is also zero by \eqref{eq:deg5obs}. By the Riemann-Roch theorem, as $Q$ is a zero dimensional sheaf, we have
\[
\dim \Ext^1 ((0,Q), (0,\cO_L(k))) - \dim \Ext^2 ((0,Q), (0,\cO_L(k))) =0.
\]
Hence it is enough to compute $\dim \Ext^1((1,\cO_C), (0,\cO_L(k)))$.

From the short exact sequence
$$\ses{(0,\cO(-4))}{(1,\cO)}{(1,\cO_C)},$$
we have
$$0\rightarrow \Ext^0((0,\cO(-4)),(0,\cO_L(k)))\stackrel{\sim}{\rightarrow} \Ext^1((1,F_{4m-2}),(0,\cO_L(k))) \rightarrow 0.$$
Recall that $k$ can be 0, 1, or 2. The first zero term is $\Hom((1,\cO),(0,\cO_L(k)))=0$ and the last term is from $\Ext^1((1,\cO),(0,\cO_L(k)))=0$, which can be seen by Proposition \ref{defcoh} because $\Ext^0(\cO,\cO_L(k))=\Hom(\CC\cdot(1), H^0(\cO_L(k)))$ and $H^1(\cO_L(k))=0$. Thus
$$\Ext^1((1,F),(0,\cO_L(k)))=H^0(\cO_L(k+4)).$$
This proves the first three of \eqref{eq2-1}.

For the last one, since $F_{3m}=\cO_C$ for some cubic curve $C$ we have an exact sequence
$$
\ses{(0,\cO(-3))}{(1,\cO)}{(1,F_{3m})},
$$
we have
$$
0\to \Ext^0((0,\cO(-3)),(0,\cO_Q))
\stackrel{\sim}{\to} \Ext^1((1,F_{3m}),(0,\cO_Q) \rightarrow \Ext^1((1,\cO),(0,\cO_Q))=0.
$$
The first zero term is clear as before. The last term comes from $\Ext^0(\cO,\cO_Q)=\Hom(\CC\cdot(1), H^0(\cO_Q))$ and $H^1(\cO_Q)=0$. Thus
$$\Ext^1((1,F_{3m}),(0,\cO_Q))=H^0(\cO_Q(3))=\CC^7.$$

Lastly we should check that the normal spaces of the flipping loci in each wall-crossing have the expected dimensions. That is, under the same notation as above, we should check
\begin{align}
&\Ext^1((0,\cO_L(2),(1,F_{4m-2})))=\CC^4, &&\Ext^1((0,\cO_L(1), (1,F_{4m-1})))=\CC^4,\\
&\Ext^1((0,\cO_L),(1,F_{4m}))=\CC^4, && \Ext^1((0,\cO_Q),(1,F_{3m}))=\CC^6. \notag
\end{align}
But these are easily checked by using a diagram of the form \eqref{eq1-4} since the extension groups of the second order are all vanished.
\end{proof}
We state a similar theorem for $\bM^{\alpha}(5,-1)$ for later use. We omit the proof since it is parallel with that of Theorem \ref{mainthm2}.
\begin{theo}\label{prop3}
There exist wall-crossing diagrams among $\bM^{\alpha}(5,-1)$
$$
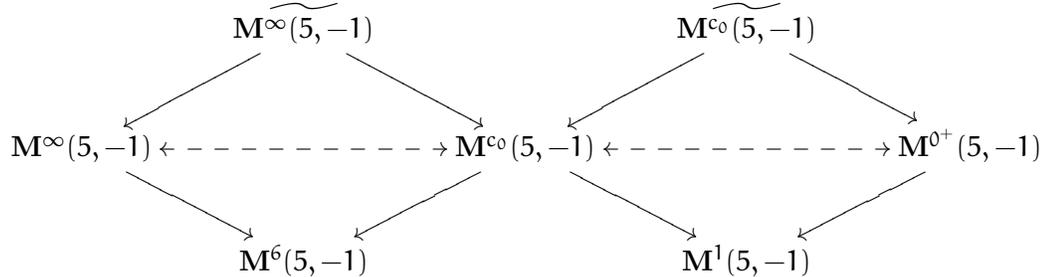
\xymatrix{& \widetilde{\bM^{\infty}(5,-1)}\ar[rd]\ar[ld] &&\widetilde{\bM^{c_0}(5,-1)}\ar[rd]\ar[ld]& \\
\bM^{\infty}(5,-1) \ar[dr]\ar@{<-->}[rr]&&\bM^{c_0}(5,-1)\ar[dr] \ar[dl]\ar@{<-->}[rr]& & \bM^{0^+}(5,-1)\ar[dl]\\
&\bM^{6}(5,-1)&&\bM^{1}(5,-1)&}
$$
such that the above arrows are all smooth blow-up morphisms and the walls occur at $\alpha=6$ and $1$. Moreover, the blow-up centers are $\PP^4$-bundle over $\bM^{0^+}(4,-2)\times \bM(1,1)$ and $\PP^3$-bundle over $\bM^{0^+}(4,-1)\times \bM(1,1)$, respectively.
\end{theo}
Similarly as in Remark \ref{rem3}, the wall-crossing loci in Theorem \ref{prop3} can be explained in a geometric way. Again, each wall-crossing is very similar to that of \cite[\S 10.3]{coskun}.
%
%

\section{Forgetful Morphisms and the Brill-Noether Loci}

Recall Proposition \ref{prop4}. If $\alpha=0^+$, there is a forgetful morphism
$$
\xi: \bM^{0^+}(d,\chi) \lra \bM(d, \chi)
$$
which forgets the section of the $0^+$-stable pair. When $\chi=1$, this map is a birational morphism and its exceptional locus is the Brill-Noether locus of the space $\bM(d, 1)$.
Let
$$
\bM(d, \chi)_k:= \{F| h^0(F)=k\}
$$
be the subscheme of $\bM(d, \chi)$, so called \emph{Brill-Noether stratum} and
$
\bM^{0^+}(d,\chi)_k:= \xi^{-1}(\bM(d, \chi)_k)
$
be the inverse image of $\bM(d, \chi)_k$ along $\xi$. We always give the reduced induced scheme structure. Then, it is immediate that $\{\bM^{0^+}(d,\chi)_k\}$ is a locally closed stratification of $\bM^{0^+}(d,\chi)$.

\begin{defi}
Let $F$ be a coherent sheaf of codimension $c$ on a smooth projective variety $X$. Then the dual sheaf is defined as $F^D=\mathcal{E}xt^c_X(F,\omega_X)$.
\end{defi}

In \cite[Theorem 13]{maican3}, it is shown that the association $F\mapsto F^D$ gives an isomorphism between the moduli spaces $\bM(d,\chi)$ and $\bM(d,-\chi)$. Moreover, we have the following.

\begin{prop}\label{prop6}
\begin{enumerate}
\item When $d$ and $\chi$ are coprime, the restriction map \\ $\bM^{0^+}(d,\chi)_k \lr \bM(d, \chi)_k$ is a Zariski locally trivial fibration with fiber $\PP^{k-1}$.
\item There is a natural isomorphism $\bM(d, \chi)_k\simeq \bM(d, -\chi)_{k-\chi}$ which sends $F$ to $F^D$.
\end{enumerate}
\end{prop}
\begin{proof}
We sketch the proof for the convenience of reader. For the detail, see Section 4.2 in \cite{choithesis}.
Let $\cF$ be a universal family of stable sheaves on $\bM(d, \chi)\times \PP^2$ and $p$ be the projection to the first factor.
Since the dimension of the zero cohomology group of stable sheaves in $\bM(d, \chi)_k$ is constantly $k$, the direct image sheaf $p_*\cF$ is locally free sheaf of rank $k$ on $\bM(d, \chi)_k$ (\cite[Corollary 12.9, III]{Hartshorne}).
Thus the projective bundle $\PP(p_*\cF^*|_{\bM(d, \chi)_k})$ with fiber $\PP^{k}$ is isomorphic to $\bM^{0^+}(d,\chi)_k$. This prove (1).

For (2), by using the local-to-global spectral sequence, one can check that if $F$ is a pure sheaf with Hilbert polynomial $dm+\chi$, then $h^0(F^D)=h^0(F)-\chi$ \cite[Corollary 6]{maican3}, \cite[Proposition 4.2.8]{choithesis}.
\end{proof}
We will see later in Lemma \ref{lem3} that by using the description of the stratification in the above proposition, one can obtain the Betti numbers of $\bM(d,1)$ from those of $\bM^{0^+}(d,1)$ and $\bM^{0^+}(d,-1)$.

When $d=4$ or $5$, through the wall-crossing analysis in previous section and the results of \cite{choithesis, maican4}, the Brill-Noether strata have the following geometric descriptions.
\begin{prop}\label{prop1}
\begin{enumerate}
\item For $d=4$ or $5$, $\bM(d,1)_k=\emptyset$ for $k\geq d-1$.
\item $\bM(4,1)_2\simeq \bM^{0^+}(4,-1)$.
\item The Brill-Noether locus of $\bM(5,1)$ consists of two components $$\bM(5,1)_2\cup\bM(5,1)_3$$ such that $\bM(5,1)_3 \simeq\bB(5,1)$ and $\overline{\bM(5,1)_2}\simeq\xi(\bM^{0^+}(5,-1))$. Moreover, $\bM(5,1)_3 \subset \overline{\bM(5,1)_2}$.
\end{enumerate}
\end{prop}
\begin{proof}
Part (1): This is \cite[\S 3]{maican} for $d=4$ and \cite[\S 3.1]{maican4} for $d=5$. A simpler proof for $d=4$ can be found in \cite[Lemma 4.6.3]{choithesis}.

Part (2): By Proposition \ref{prop6}, $$\bM(4,1)_2\simeq\bM(4,-1)_1\simeq \bM^{0^+}(4,-1)_1.$$ The last space $\bM^{0^+}(4,-1)_1=\bM^{0^+}(4,-1)$ because of Proposition \ref{prop6}.(2) and part (1).

Part (3): By \cite[Proposition 3.1.5 and Proposition 3.3.3]{maican4}, the general points in $\bM(5,1)_2$ (resp. $\bM(5,1)_3$) consist of stable sheaves of the form $\cO_C(2)(-p_1-p_2-p_3-p_4)$ (resp. $\cO_C(1)(p)$) for four points $p_i$ (resp. a point $p$) in general position on smooth quintic curves. Then, obviously, $\bM(5,1)_3\simeq \bB(5,1)\simeq \bM^{0^+}(5,-4)$. Also, as we have seen in Theorem \ref{prop3}, $\bM^{0^+}(5,-1)$ is obtained from the moduli space $\bB(5,4)$ by two times wall-crossings where the pairs in the flipping locus are supported on reducible quintic curves. Hence the general sheaves are of the form
\begin{equation}\label{eq3-1}
\cO_C(p_1+p_2+p_3+p_4)
\end{equation}
under the above condition. Through the composition of the forgetful and the dual map
$$
\xi:\bM^{0^+}(5,-1)\lr \bM(5,-1)=\bM(5,1),
$$
theses sheaves in \eqref{eq3-1} exactly correspond to the general points in $\bM(5,1)_2$ \cite[Proposition 3.3.3]{maican4}. Since $\bM(5,1)_2$ and $\xi(\bM^{0^+}(5,-1))$ are both irreducible, we get $\xi(\bM^{0^+}(5,-1))=\overline{\bM(5,1)_2}$. The last inclusion is easily proved by deforming four general points $p_i$ into colinear ones \cite[Proposition 3.3.4]{maican4}.
\end{proof}
\begin{prop}\label{prop5}
The forgetful morphism $ \xi: \bM ^{0^+}(4,1) \lra \bM(4,1)$ is a smooth blow-up along the moduli space $\bB(4,1)=\bM(4,1)_2$.
\end{prop}
\begin{proof}
Let $\mathcal{F}$ be a universal family of stable sheaves on
$\bB(4,1) \times \PP^2$ and $p$ be the projection into
the first factor $\bB(4,1)$. By (2) in Proposition \ref{prop1}, there is a morphism
$$
\bB(4,1)\simeq \bM(4,1)_2 \subset \bM(4,1),
$$
where the stable pair $(1,F)$ corresponds to its dual $F^{D}$.
The tangent map of this morphism is presented by
$$
0\rightarrow\Ext^1((1,F),(1,F))=T_{(1,F)}\bB(4,1)\rightarrow \Ext^1(F,F)\simeq \Ext^1(F^D,F^D)=T_F \bM(4,1)\rightarrow H^1(F) \rightarrow 0,
$$
where the last term is zero by (2) of Remark \ref{rem1}. Therefore
$$
N_{\bB(4,1)/\bM(4,1),F}\simeq H^1(F)\simeq H^0(F^D)^*,
$$
where the last isomorphism is given by \cite[Proposition 4.2.8]{choithesis}.
Since this isomorphism is canonical, we can say that the normal bundle of $\bB(4,1)$ in $\bM(4,1)$ is isomorphic with the dual of the direct image sheaf of the projection $p$
$$
N_{\bB(4,1)/\bM(4,1)}\simeq (p_*\mathcal{F}^D)^*.
$$
Since $H^1(F)$ is constant fiberwisely, the direct image sheaf $p_*\mathcal{F}^D$ is a locally free sheaf of rank two
on $\bB(4,1)$. So the projective bundle
$$\mathbf{P}:=\PP((p_*\mathcal{F}^D)^*)$$
is a $\PP^1$-bundle over $\bB(4,1)$. As we have seen in Proposition \ref{prop6}, there is a closed embedding
$$i:\mathbf{P}\hookrightarrow\bM ^{0^+}(4,1)$$
such that there is a commutative diagram
$$
\xymatrix{ \mathbf{P}\ar@{^(->}[r]\ar[d]&\bM^{0^+}(4,1)\ar[d]^{\xi}\\
\bB(4,1)\simeq \bM(4,1)_2\ar@{^(->}[r]&\bM(4,1).
}
$$
Obviously, the image of the $\mathbf{P}$ by $i$ is the exceptional divisor of $\xi$ and thus the morphism $\xi$ is a smooth blow-up morphism.
\end{proof}

\begin{rema}
By using part (3) in Proposition \ref{prop1}, one can easily check that the forgetful map
$$
\xi: \bM^{0^+}(5,1)\lr \bM(5,1)
$$
is also a divisiorial contraction but not a smooth one. We remark that the contracted divisor $\bM^{0^+}(5,1)_{k\geq2} $ is an irreducible variety, which can be geometrically proved by considering the wall-crossings of the moduli spaces $\bM^{\alpha}(5,-1)$ with Hilbert polynomial $5m-1$ (Theorem \ref{prop3}).
\end{rema}
\section{Betti Numbers}\label{compbetti}
In this section, we present two corollaries of Theorem \ref{mainthm1} and Theorem \ref{mainthm2}. By using the wall-crossing formula, one can easily obtain all Betti numbers of Simpson
spaces $ \bM(4,1)$ and $ \bM(5,1)$.
For a variety $X$, let us define the \emph{Poincar\'e polynomial} of $X$ by
\[ P(X)=\sum_{i\ge 0} \dim_{\QQ} H^{i}(X,\QQ) q^{i/2}. \]
Since odd cohomology groups of moduli spaces of our interests always vanish, $P(X)$ is a \emph{polynomial}.
\begin{lemm}\label{lem3}
For any degree $d\ge 1$, we have
\[ P(\bM(d,1))= P(\bM^{0^+}(d,1))- q P(\bM^{0^+}(d,-1)). \]
\end{lemm}
\begin{proof}
This is the Poincar\'e polynomial version of \cite[Proposition 4.2.9]{choithesis}. 
By Proposition \ref{prop6}, we have
\begin{align*}
P(\bM^{0^+}&(d,1))- q P(\bM^{0^+}(d,-1)) \\
& = \sum_{k\ge 1} P(\PP^{k-1})\cdot P(\bM(d,1)_k) - q \sum_{k\ge 1} P(\PP^{k-1})\cdot P(\bM(d,-1)_k) \\
& = \sum_{k\ge 1} P(\PP^{k-1})\cdot P(\bM(d,1)_k) - q \sum_{k\ge 1} P(\PP^{k-1})\cdot P(\bM(d,1)_{k+1}) \\
& = \sum_{k\ge 1} (P(\PP^{k-1}) - q P(\PP^{k-2}))\cdot P(\bM(d,1)_{k}) \\
& = \sum_{k\ge 1} P(\bM(d,1)_k) \\
&=  P(\bM(d,1)),
\end{align*}
where $P(\PP^{-1}):=0$.
\end{proof}

\begin{coro}\label{betti4}
The Poincar\'e polynomial of the Simpson space $ \bM(4,1)$ is given by
\begin{multline*}
1+2q+6q^2+10q^3+14q^4+15q^{5}+16q^{6}+16q^{7}+16q^{8} \\
+16q^{9}+16q^{10}+16q^{11}+15q^{12}+14q^{13}+10q^{14}+6q^{15}+2q^{16}+q^{17}.
\end{multline*}
\end{coro}
\begin{proof}
We know $\bM^{\infty}(4,1)\simeq \bB(4,3)$ is a $\PP^{11}$-bundle over $\mbox{Hilb}^3(\PP^2)$. In \cite{ele}, the Poincar\'e polynomial of $\mbox{Hilb}^3(\PP^2)$ is given by
$$P(H ilb^{3}(\PP^2))=1+2q+5q^2+6q^3+5q^4+2q^5+q^6.$$

Now, by the wall-crossing of Theorem \ref{mainthm1}, we obtain the Poincar\'e polynomial of $\bM^{0^+}(4,1)$.
\begin{align*}
  P(&\bM^{0^+}(4,1)) = P(\PP^{11})\cdot P(\mbox{Hilb}^{3}(\PP^2)) - (P(\PP^3)-P(\PP^2))P(\PP^9\times\PP^2)\\
  &=\frac{1-q^{12}}{1-q}\cdot(1+2q+5q^2+6q^3+5q^4+2q^5+q^6)- q^3\cdot\frac{1-q^{10}}{1-q}\cdot \frac{1-q^3}{1-q}
\end{align*}
On the other hand, $\bM^{0^+}(4,-1)\simeq\bM^{\infty}(4,-1)\simeq \bB(4,1)$ is a $\PP^{13}$-bundle over $\PP^2$.
$$P(\bM^{0^+}(4,-1))= P(\PP^{13})\cdot P(\PP^2)=\frac{1-q^{14}}{1-q}\cdot\frac{1-q^3}{1-q} $$
Therefore, we obtain the Poincar\'e polynomial by Lemma \ref{lem3}.
\end{proof}
\begin{coro}\label{coro2}
The Poincar\'e polynomial of the Simpson space $\bM(5,1)$ is given by
\begin{multline*}
1 + 2q+6q^{2}+13q^{3}+26q^{4}+45q^{5}+68q^{6}+87q^{7}+100q^{8}+107q^{9} \\
+111q^{10}+112q^{11}+113q^{12}+113q^{13}+113q^{14}+112q^{15}+111q^{16}+107q^{17} \\
+100q^{18}+87q^{19}+68q^{20}+45q^{21}+26q^{22}+13q^{23}+6q^{24}+2q^{25}+q^{26}.
\end{multline*}
\end{coro}
\begin{proof}
From Theorem \ref{mainthm2}, we have
\begin{align*}
  P(\bM^{0^+}(5,1))= P(\bB(5,6)) & +(P(\PP^3)-P(\PP^6))\cdot P(\bB(4,0))\cdot P(\PP^{2})\\
  &+(P(\PP^3)-P(\PP^5))\cdot P(\bB(4,1))\cdot P(\PP^{2})\\
  &+(P(\PP^3)-P(\PP^4))\cdot P(\bB(4,2))\cdot P(\PP^{2})\\
  &+(P(\PP^5)-P(\PP^6))\cdot P(\bB(3,0))\cdot P(\PP^5).
\end{align*}
By Theorem \ref{prop3}, we have
\begin{align*}
P(\bM^{0^+}(5,-1))= P(\bB(5,4)) & +(P(\PP^3)-P(\PP^4))\cdot P(\bB(4,0))\cdot P(\PP^{2})\\
&+(P(\PP^3)-P(\PP^3))\cdot P(\bB(4,1))\cdot P(\PP^{2}). \
\end{align*}
Thus we obtain the result by Lemma \ref{lem3}.
\end{proof}
\begin{rema}
\begin{enumerate}
\item The results in Corollary \ref{betti4} and \ref{coro2} coincide with predictions in physics \cite{hkk} by B-model calculation. Also, these are consistent with the results in \cite{sahin, choithesis, yuan2, cm, maican5}.
\item It has been conjectured that the topological Euler characteristics of $\bM(d,1)$ are equal to the genus zero Gopakumar-Vafa invariants up to sign \cite{katz_gv}. By specializing the Poincar\'e polynomials to $q=1$, we can see that the topological Euler characteristics of moduli spaces $\bM(4,1)$ and $\bM(5,1)$ are $192$ and $1675$ respectively, which matches with the prediction in physics \cite{KKV}. Moreover, our approach provides another explanation to the correction terms in the computation of Goparkumar-Vafa invariants in \cite{KKV}. See \cite{ckk} for more details.
\end{enumerate}
\end{rema}


\section{Euler Characteristic of the Space $\bM^{0^+}(4,3)$}
In this section, we calculate the Euler characteristic $\chi(\bM^{0^+}(4,3))$ of the moduli space $\bM^{0^+}(4,3)$ via the wall-crossing technique we have been using. We will encounter a new type of wall where strictly semistable pairs can have a Jordan-H\"{o}lder filtration of length 3. 

The possible types of strictly semistable pairs are as follows.
\begin{center}
\begin{tabular}{|l|p{6cm}|}
\hline
\multicolumn{2}{|l|}{$(d,\chi)=(4,3)$} \\
\hline
$\alpha$
&
$(1,P(F))\oplus (0,P(F'))$\\
\hline
9&
$(1,(3,0)\oplus (0,(1,3))$\\
\hline
5&
$(1,(3,1))\oplus (0,(1,2))$\\
\hline
1&
$(1,(3,2))\oplus (0,(1,1))$\\
\hline
$1$&
$(1,(2,1))\oplus (0,(2,2))$\\
\hline
$1$&
$(1,(2,1))\oplus (0,(1,1))\oplus (0,(1,1))$\\
\hline
\end{tabular}
\end{center}

At the walls $\alpha=9$ and $\alpha=5$, semistable pairs can only have a length two Jordan-H\"{o}lder filtration. Similarly as
before, it can be shown that there are flip diagrams at these walls. We may compute how the Euler characteristic changes as we cross these wall:
\[
\chi(\bM^\infty(4,3))= \chi(\bB(4,5))=1080.
\]
\begin{align*}
\chi(\bM^{5<\alpha<9}(4,3))& = \chi(\bM^\infty(4,3)) +(\chi(\PP^2)-\chi(\PP^5))\cdot \chi(\bB(3,0))\cdot \chi(\PP^{2})\\
& =1080 - 3\cdot 10\cdot 3 = 990 .
\end{align*}
\begin{align*}
\chi(\bM^{1<\alpha<5}(4,3))& = \chi(\bM^{5<\alpha<9}(4,3)) +(\chi(\PP^2)-\chi(\PP^4))\cdot \chi(\bB(3,1))\cdot \chi(\PP^{2})\\
& =990 - 2\cdot 27\cdot 3 = 828 .
\end{align*}

At the wall $\alpha=1$, strictly semistable pairs can split into either $(1,(3,2))\oplus (0,(1,1))$ or $(1,(2,1))\oplus (0,(2,2))$. Moreover it is possible that the component $(1,(3,2))$ in the first decomposition and $(0,(2,2))$ in the second decomposition are strictly semistable, so they may split further into stable pieces, which gives the last case $(1,(2,1))\oplus (0,(1,1))\oplus (0,(1,1))$.

The following lemmas are elementary.
\begin{lemm}
Suppose $\alpha>1$. A pair $(1,F)$ given by an extension
\begin{equation}\label{eq:A1+}
\ses{\Lambda':=(0,\cO_L)}{(1,F)}{\Lambda:=(1,F_{3m+2})}
\end{equation}
is $\alpha$-stable if and only if the short exact sequence is nonsplit and $\Lambda$ is $\alpha$-stable.
\end{lemm}

The analogous statement for $\alpha<1$ also holds.

\begin{lemm}
Suppose $\alpha<1$. A pair $(1,F)$ given by an extension
\begin{equation}\label{eq:A1-}
\ses{\Lambda:=(1,F_{3m+2})}{(1,F)}{\Lambda':=(0,\cO_L)}
\end{equation}
is $\alpha$-stable if and only if the short exact sequence is nonsplit and $\Lambda$ is $\alpha$-stable.
\end{lemm}

Meanwhile, for the other type of splitting $(1,(2,1))\oplus (0,(2,2))$, we only have one direction.
\begin{lemm}
Suppose $\alpha>1$ and a pair $(1,F)$ is given by an extension
\begin{equation}\label{eq:B1+}
\ses{\Lambda_2':=(0,F_{2m+2})}{(1,F)}{\Lambda_2:=(1,F_{2m+1})}.
\end{equation}
If $\Lambda_2$ and ${\Lambda_2'}$ are $\alpha$-stable and \eqref{eq:B1+} is nonsplit, then $(1,F)$ is $\alpha$-stable.
\end{lemm}

\begin{lemm}
Suppose $\alpha<1$ and a pair $(1,F)$ is given by an extension
\begin{equation}\label{eq:B1-}
\ses{\Lambda_2:=(1,F_{2m+1})}{(1,F)}{\Lambda_2':=(0,F_{2m+2})}.
\end{equation}
If $\Lambda_2$ and ${\Lambda_2'}$ are $\alpha$-stable and \eqref{eq:B1-} is nonsplit, then $(1,F)$ is $\alpha$-stable.
\end{lemm}
On crossing the wall, pairs in \eqref{eq:A1+} are replaced by pairs in \eqref{eq:A1-} and pairs in \eqref{eq:B1+} are by pairs in \eqref{eq:B1-}. 

For $\alpha>1$ (resp. $\alpha<1$), we denote by $A^+$ (resp. $A^-$) the space of $\alpha$-stable pairs which fit into \eqref{eq:A1+} (resp. \eqref{eq:A1-}), and by $B^+$ (resp. $B^-$) the space of $\alpha$-stable pairs which fit into \eqref{eq:B1+} (resp. \eqref{eq:B1-}). By the following lemma, a special consideration is needed for pairs in the intersection $A^+\cap B^+$ or $A^-\cap B^-$ such that $F_{2m+2}$ in \eqref{eq:B1+} or \eqref{eq:B1-} is a direct sum $\cO_{L_1}\oplus \cO_{L_2}$. We denote the space of such pairs by $C^+$ or $C^-$ respectively. In what follows, the wall-crossing on $A$ means the difference in Euler characteristic $\chi(A^-) - \chi(A^+)$, etc.

\begin{lemm}
The $\alpha$-stable pairs $(1,F)$ in $A^+-C^+$, $A^--C^-$, $B^+-C^+$, or $B^--C^-$ are in one-to-one correspondence with the isomorphism classes of the corresponding exact sequences.
\end{lemm}
\begin{proof}
We shall sketch a proof for $\alpha$-stable pair in $A^+-C^+$. Other cases are similar.
We need to show that if an $\alpha$-stable pair $(1,F)$ in $A^+-C^+$ fits into two exact sequences
\[ \ses{(0,\cO_L)}{(1,F)}{(1,F_{3m+2})} \]
and
\[ \ses{(0,\cO_{L'})}{(1,F)}{(1,F_{3m+2}')}, \]
then they are isomorphic.

Since $\cO_L$ and $\cO_{L'}$ are subsheaves of $F$, either $\cO_{L}\oplus\cO_{L'}$ is a subsheaf of $F$ or $L=L'$, where the former case is when $(1,F)$ is in $C^+$. Hence, $L=L'$ and the above exact sequences are clearly isomorphic.
\end{proof}

A careful consideration is needed for the case $F_{2m+2}$ is a direct sum $\cO_{L_1}\oplus \cO_{L_2}$ which has a bigger automorphism group.
\begin{enumerate}
\item When $L_1$ and $L_2$ are distinct. In this case, the automorphism group of $\cO_{L_1}\oplus \cO_{L_2}$ is $\CC^*\times\CC^*$. When $\alpha>1$, $\alpha$-stable pairs as in \eqref{eq:B1+} form the space
\begin{equation}\label{eq:dist1+}\frac{\Ext^1((1,F_{2m+1}),(0,\cO_{L_1}\oplus \cO_{L_2}))^{\mbox{st}}}{\CC^*\times\CC^*}.\end{equation}
The superscript ``st'' means taking extensions corresponding to stable pairs. It is easy to see that stable extensions are those which do not factor through extensions $\Ext^1((1,F_{2m+1}),(0,\cO_{L_i}))$ for $i=1$ or $2$. Therefore, \eqref{eq:dist1+} becomes $$\PP(\Ext^1((1,F_{2m+1}),(0,\cO_{L_1})))\times\PP(\Ext^1((1,F_{2m+1}),(0,\cO_{L_2})))\simeq \PP^2\times\PP^2.$$

Similarly, when $\alpha<1$, we can see that $\alpha$-stable pairs as in \eqref{eq:B1-} form the space
\begin{equation}\label{eq:dist1-}
\PP((0,\cO_{L_1}),\Ext^1((1,F_{2m+1})))\times\PP((0,\cO_{L_2}),\Ext^1((1,F_{2m+1})))\simeq \PP^1\times\PP^1.
\end{equation}
\item When $L_1=L_2=L$. The automorphism group of $\cO_{L}\oplus \cO_{L}$ is $GL(2,\CC)$. When $\alpha>1$, $\alpha$-stable pairs as in \eqref{eq:B1+} form the space
\begin{equation}\label{eq:same1+}\frac{\Ext^1((1,F_{2m+1}),(0,\cO_{L}\oplus \cO_{L}))^{\mbox{st}}}{GL(2,\CC)}.\end{equation}
We write $\Ext^1((1,F_{2m+1}),(0,\cO_{L}\oplus \cO_{L}))$ as $\Hom(\CC^2,\Ext^1((1,F_{2m+1}),(0,\cO_{L})))$. Then stable pairs correspond to rank 2 maps in the latter space. Hence, \eqref{eq:same1+} becomes the Grassmannian $$Gr(2, \Ext^1((1,F_{2m+1}),(0,\cO_{L})))\simeq \PP^2.$$
Similarly, when $\alpha<1$, we can see that $\alpha$-stable pairs as in \eqref{eq:B1-} form the space
$$Gr(2, \Ext^1((0,\cO_{L}),(1,F_{2m+1})))= \mbox{pt}.$$
\end{enumerate}

We compute the wall-crossing by the decomposition $A\cup B= (B-A) \sqcup(A-C)\sqcup C$. It is clear that each set is locally closed.

The wall-crossing on $B-A$ is given by
\[
(\chi(\PP^3)-\chi(\PP^5))\cdot \chi(\bB(2,0)) \cdot \chi(\bM^s(2,2)),
\]
where $\bM^s(2,2)$ denote the space of stable sheaves, that is, we exclude strictly semistable sheaves where a further splitting can occur. It is easy to see that $\bM^s(2,2)\simeq \PP^5-V$, where $\PP^5$ is the space of conics and $V$ is the space of degenerate conics. Since the Euler characteristic $\chi(\bM^s(2,2))$ is zero, the wall-crossing on $B-A$ is zero.

Let $D\subset V$ be the diagonal. As discussed above, the wall-crossing on $C$ is given by
\begin{equation}
(\chi(\PP^1\times\PP^1)-\chi(\PP^2\times\PP^2))\cdot\chi(\bB(2,0))\cdot \chi(V-D) + (\chi(\mbox{pt})-\chi(\PP^2))\cdot\chi(\bB(2,0))\cdot \chi(D).
\end{equation}

It remains to compute the wall-crossing on $A-C$. 
Suppose $\alpha>1$ and let $(1,F)$ be an $\alpha$-stable pair in $C$. Then $(1,F)$ is given by an exact sequence
\[
\ses{(0,\cO_L)}{(1,F)}{(1,F_{3m+2})}
\]
Since $(1,F)$ is also in $B$, $(1,F_{3m+2})$ in the above sequence fits into an exact sequence
\begin{equation}\label{eq:CA1+}
\ses{(0, \cO_{L'})}{(1,F_{3m+2})}{(1,F_{2m+1})}.
\end{equation}
From \eqref{eq:CA1+}, we have
\[
\Ext^1((1,F_{3m+2}),(0,\cO_L)) \stackrel{\epsilon}{\to} \Ext^1((0, \cO_{L'}),(0,\cO_L))\to 0.
\]
One can see that $(1,F)$ is in $C$ if and only if its image by $\epsilon$ in $\Ext^1((0, \cO_{L'}),(0,\cO_L))$ is zero. We have \[\Ext^1((1,F_{3m+2}),(0,\cO_L))\simeq \CC^4,\]
and
\[\Ext^1((0, \cO_{L'}),(0,\cO_L))\simeq \begin{cases}
\CC & \text{if }L \neq L',\\
\CC^2 & \text{if }L=L'.\\
\end{cases} \]
So, the Euler characteristic of the set $A^+-C^+$ is
\begin{align*}
\chi(\PP^3)\cdot \chi(\bM(1,1))\cdot \chi(\bB(3,2)) &- \chi(\PP^2) \cdot \chi(\bM(1,1))\cdot (\chi(\PP^2)\cdot(\chi(\bB(2,0))\chi(\bM(1,1)-\mbox{pt})))\\ & - \chi(\PP^1) \cdot \chi(\bM(1,1))\cdot (\chi(\PP^2)\cdot(\chi(\bB(2,0))\chi(\mbox{pt}))).
\end{align*}

The computation for $\alpha<1$ is similar: the only difference is
\[\Ext^1((0,\cO_L),(1,F_{3m+2}))\simeq \CC^3.\]
Then, the Euler characteristic of the set $A^--C^-$ is
\begin{align*}
\chi(\PP^2)\cdot \chi(\bM(1,1))\cdot \chi(\bM^{0^+}(3,2)) &- \chi(\PP^1) \cdot \chi(\bM(1,1))\cdot (\chi(\PP^1)\cdot(\chi(\bB(2,0))\chi(\bM(1,1)-\mbox{pt})))\\ & - \chi(\mbox{pt}) \cdot \chi(\bM(1,1))\cdot (\chi(\PP^1)\cdot(\chi(\bB(2,0))\chi(\mbox{pt}))).
\end{align*}
In conclusion, we have
\[ \chi(\bM^{0^+}(4,3)) = \chi(\bM^{1<\alpha<5}(4,3)) + \mbox{wall-crossing} = 828-252 = 576. \]

\medskip
This coincides with our previous calculation. For all $F \in \bM(4,3)$, we have $H^0(F)=3$ and $H^1(F)=0$ \cite[Lemma 4.2.4]{choithesis}. Hence by Proposition \ref{prop6}, we conclude that $\bM^{0^+}(4,3)$ is a $\PP^2$-bundle over $\bM(4,3)$. We also have $\bM(4,3)\simeq\bM(4,1)$, whose Euler characteristic is $192$ by Corollary \ref{betti4}. Therefore, the Euler characteristic of $\bM^{0^+}(4,3)$ is $3\times 192 = 576$.

\begin{rema}
The Poincar\'e polynomial of $\bM^{0^+}(4,3)$ cannot be computed through our stratifications since the wall-crossing terms from \eqref{eq:dist1+} and \eqref{eq:dist1-} are not Zariski locally trivial fiberations over $V-D$, nevertheless the computation for Euler characteristic still holds because $V-D$ is path-connected.
\end{rema}

\bibliographystyle{amsplain}

\end{document}